\documentclass{article}
\usepackage[english]{babel}
\usepackage[latin1]{inputenc}
\usepackage{latexsym}
\usepackage{amssymb,amsmath,dsfont}
\usepackage{mathrsfs}
\usepackage{subfigure}
\usepackage{graphicx,float}
\usepackage{color}
\usepackage{url}
\usepackage{authblk}

\usepackage{array}
\usepackage[table]{xcolor}
\usepackage{amsthm}
\usepackage{a4}
\usepackage{fullpage}
\usepackage[unicode]{hyperref}
\hypersetup{
    bookmarks=true,         
    unicode=true,          
    pdftoolbar=true,        
    pdfmenubar=true,        
    pdffitwindow=true,      
    pdftitle={Wave delay},    
    pdfauthor={Stéphane Gerbi and Belkacem Said-Houari},     
    pdfsubject={},   
    pdfcreator={},   
    pdfproducer={}, 
    pdfkeywords={}, 
    pdfnewwindow=true,      
    colorlinks=true,       
    linkcolor=blue,  
    citecolor=blue,        
    filecolor=blue,      
    urlcolor=blue           
}
\newtheorem{theorem}{Theorem}[section]

\newtheorem{lemma}{Lemma}[section]

\newtheorem{remark}{Remark}[section]
\newtheorem*{ack}{Acknowledgments}

\newcommand{\findem}{\begin{flushright} \vspace*{-0.25cm} $\square$ \end{flushright}}
 \title{Existence and exponential stability of a damped wave equation with
dynamic boundary conditions and a delay term}
\author[1]{{\sc{Stéphane Gerbi}}\thanks{\url{stephane.gerbi@univ-savoie.fr}}}
\author[2]{{\sc{Belkacem Said-Houari}}\thanks{\url{saidhouarib@yahoo.fr}}}
\affil[1]{\small{Laboratoire de Mathématiques, UMR 5127 - CNRS and Université de Savoie, \par73376 Le Bourget-du-Lac Cedex, France.}}
\affil[2]{\small{Division of Mathematical and Computer Sciences and Engineering, \par King Abdullah University of Science and Technology (KAUST), Thuwal, Saudi Arabia}}
\begin{document}
\date{}
\maketitle

\begin{abstract}
In this paper we consider a multi-dimensional wave equation with dynamic
boundary conditions related to the Kelvin-Voigt damping and a delay term
acting on the boundary. If the weight of the delay term in the feedback is
less than the weight of the term without delay or if it is greater under an
assumption between the damping factor, and the difference of the two
weights, we prove the global existence
of the solutions. Under the same assumptions, the exponential stability of the system is proved using an
appropriate Lyapunov functional. More precisely, we  show that even when
the weight of the delay is greater than the weight of the damping in the
boundary conditions, the strong damping term still provides exponential
stability for the system.
\end{abstract}
\textbf{Keywords}: Damped wave equations, boundary delay;  global solutions;  exponential stability;  Kelvin-Voigt damping;  dynamic boundary conditions.
\section{Introduction}

In this paper we consider the following linear damped wave equation with
dynamic boundary conditions and a delay boundary term:
\begin{equation}
\left\{
\begin{array}{ll}
u_{tt}-\Delta u-\alpha \Delta u_{t}=0, & x\in \Omega ,\ t>0 \,,\,\\[0.1cm]
u(x,t)=0, & x\in \Gamma _{0},\ t>0 \,,\,\\[0.1cm]
u_{tt}(x,t)=-\left( \displaystyle\frac{\partial u}{\partial \nu }(x,t)+\frac{\alpha \partial u_{t}}{\partial \nu }(x,t)+\mu_{1}u_{t}(x,t)+
\mu_{2}u_{t}(x,t-\tau )\right) & x\in \Gamma_{1},\ t>0 \,,\,\\[0.1cm]
u(x,0)=u_{0}(x) & x\in \Omega \,,\,\\[0.1cm]
u_{t}(x,0)=u_{1}(x) & x\in \Omega \,,\, \\[0.1cm]
u_{t}(x,t-\tau)=f_{0}(x,t-\tau ) & x\in\Gamma_{1},\, t\in (0,\tau)\,,\,
\end{array}%
\right.  \label{ondes}
\end{equation}
where $u=u(x,t)\,,\,t\geq 0\,,\,x\in \Omega \,,\,\Delta $ denotes the
Laplacian operator with respect to the $x$ variable, $\Omega $ is a regular
and bounded domain of $\mathbb{R}^{N}\,,\,(N\geq 1)$, $\partial\Omega
~=~\Gamma _{0}~\cup ~\Gamma _{1}$, $mes(\Gamma _{0})>0,$ $\overline{\Gamma}_{0}\cap \overline{\Gamma}_{1}=\varnothing $ 
and $\displaystyle\frac{\partial }{\partial \nu }$
denotes the unit outer normal derivative, $\,\alpha,\,\mu _{1}$ and $\mu
_{2} $ are positive constants. Moreover, $\tau >0$ represents the time delay
and $u_{0}\,,\,u_{1},~f_{0}$ are given functions belonging to suitable
spaces that will be precised later.

This type of problems arise (for example) in modelling of longitudinal
vibrations in a homogeneous bar in which there are viscous effects. The term
$\Delta u_t$, indicates that the stress is proportional not only to the
strain, but also to the strain rate. See \cite{CSh_76}. From the
mathematical point of view, these problems do not neglect acceleration terms
on the boundary. Such type of boundary conditions are usually called \textit{%
dynamic boundary conditions}. They are not only important from the
theoretical point of view but also arise in several physical applications.
For instance in one space dimension, problem (\ref{ondes}) can modelize the
dynamic evolution of a viscoelastic rod that is fixed at one end and has a
tip mass attached to its free end. The dynamic boundary conditions
represents the Newton's law for the attached mass, (see \cite{BST64,AKS96,
CM98} for more details). In the two dimension space, as showed in \cite{G06}
and in the references therein, these boundary conditions arise when we
consider the transverse motion of a flexible membrane $\Omega $ whose
boundary may be affected by the vibrations only in a region. Also some
dynamic boundary conditions as in problem (\ref{ondes}) appear when we
assume that $\Omega $ is an exterior domain of $\mathbb{R}^{3} $ in which
homogeneous fluid is at rest except for sound waves. Each point of the
boundary is subjected to small normal displacements into the obstacle (see
\cite{B76} for more details). This type of dynamic boundary conditions are
known as acoustic boundary conditions.

In the absence of the delay term (i.e. $\mu_{2} =0$) problem (\ref{ondes})
 has been investigated by many authors
in recent years (see, e.g., \cite{GS082}, \cite{GS08}, \cite{G94_1}, \cite%
{G94_2}, \cite{P08}, \cite{PS04}).

Among the early results dealing with the \textit{dynamic boundary conditions} are
those of Grobbelaar-Van Dalsen \cite{GV94,GV96} in which the author has made
contributions to this field.

In \cite{GV94} the author introduced a model which describes the damped
longitudinal vibrations of a homogeneous flexible horizontal rod of length $%
L $ when the end $x = 0$ is rigidly fixed while the other end $x=L$ is free
to move with an attached load. This yields to a system of two second order
equations of the form
\begin{equation}  \label{Vad_1}
\left\{
\begin{array}{ll}
u_{tt}-u_{xx}-u_{txx}=0, & x\in (0,L),\,t>0, \\[0.1cm]
u(0,t)=u_{t}(0,t)=0, & t>0, \\[0.1cm]
u_{tt}(L,t)=-\left[ u_{x}+u_{tx}\right] (L,t), & t>0, \\[0.1cm]
u\left( x,0\right) =u_{0}\left( x\right) ,u_{t}\left( x,0\right)
=v_{0}\left( x\right), & x\in (0,L), \\
u\left( L,0\right) =\eta ,\qquad u_{t}\left( L,0\right) =\mu. &
\end{array}
\right.
\end{equation}
By rewriting problem (\ref{Vad_1}) within the framework of the abstract
theories of the so-called $B$-evolution theory, an existence of a unique
solution in the strong sense has been shown. An exponential decay result was
also proved in \cite{GV96} for a problem related to (\ref{Vad_1}), which
describe the weakly damped vibrations of an extensible beam. See \cite{GV96}
for more details.

Subsequently, Zang and Hu \cite{ZH07}, considered the problem
\begin{equation*}
   u_{tt}-p\left( u_{x}\right) _{xt}-q\left( u_{x}\right) _{x}=0,\qquad x\in \left(0,1\right) ,\,t>0
\end{equation*}
with
\begin{equation*}
 u\left( 0,t\right) =0,\qquad   p\left( u_{x}\right) _{t}+q\left( u_{x}\right) \left( 1,t\right)
+ku_{tt}\left( 1,t\right) =0,\, t\geq 0.
\end{equation*}
By using the Nakao inequality, and under appropriate conditions on $p$ and $%
q $, they established both exponential and polynomial decay rates for the
energy depending on the form of the terms $p$ and $q$.

It is clear that in the absence of the delay term and for $\mu_1=0$, problem
(\ref{Vad_1}) is the one dimensional model of (\ref{ondes}). Similarly, and
always in the absence of the delay term, Pellicer and Sol\`{a}-Morales \cite%
{PS04} considered the one dimensional problem of (\ref{ondes}) as an
alternative model for the classical spring-mass damper system, and by using
the dominant eigenvalues method, they proved that their system
has the classical second order differential equation
\begin{equation*}
m_{1}u^{\prime \prime }(t)+d_{1}u^{\prime }(t)+k_{1}u(t)=0,
\end{equation*}%
as a limit, where the parameter $m_{1}\,,\,d_{1}\mbox{ and }k_{1}$ are
determined from the values of the spring-mass damper system. Thus, the
asymptotic stability of the model has been determined as a consequence of
this limit. But they did not obtain any rate of convergence. See also  \cite{P08,PS08} for related results.

Recently, the present authors studied in \cite{GS08} and \cite{GS082} a
more general situation of (\ref{ondes}). They considered problem (\ref%
{ondes}) with $\mu_2=0$, a nonlinear damping of the form $g\left(u_{t}\right)~=~\left%
\vert u_{t}\right\vert^{m-2}u_{t}$ instead of $\mu_{1}u_{t}$ and a nonlinear
source term $f\left(u\right)~=~\left\vert u\right\vert^{p-2}u_{t}$ in the
right hand side of the first equation of problem (\ref{ondes}). A local
existence result was obtained by combining the Faedo-Galerkin method with
the contraction mapping theorem. Concerning the asymptotic behavior, the
authors showed that the solution of such problem is unbounded and grows up
exponentially when time goes to infinity if the initial data are large
enough and the damping term is nonlinear. The blow up result was shown when
the damping is linear. Also, we proved in \cite{GS082} that under some
restrictions on the exponents $m$ and $p$, we can always find initial data
for which the solution is global in time and decay exponentially to zero.

The main difficulty of the problem considered is related to the non ordinary
boundary conditions defined on $\Gamma _{1}$. Very little attention has been
paid to this type of boundary conditions.  We mention only a few particular
results in the one dimensional space \cite{GV99,PS04,DL02,K92}.

The purpose of this paper is to study problem (\ref{ondes}), in which a
delay term acted in the dynamic boundary conditions. In recent years one
very active area of mathematical control theory has been the investigation
of the delay effect in the stabilization of hyperbolic systems and many
authors have shown that delays can destabilize a system that is
asymptotically stable in the absence of delays (see \cite{DLP86} for more
details).

In \cite{NP06}, Nicaise and Pignotti examined the wave equation with a linear
boundary damping term with a delay. Namely, they looked to the following
problem:
\begin{equation}\label{Wave_equation}
    u_{tt}-\Delta u=0, \quad x\in \Omega ,\, t>0,
\end{equation}
where $\Omega$ is a bounded domain with smooth boundary $\partial\Omega=\Gamma _{0}\cup \Gamma _{1}$. On $\Gamma _{0}$, they
considered the Dirichlet boundary conditions. While on $\Gamma _{1}$ they assumed the following boundary conditions:
\begin{equation}\label{Boundary_Nicaise}
 \frac{\partial u}{\partial \nu }(x,t)=\mu_{1}u_{t}(x,t)+\mu_{2}u_{t}(x,t-\tau ),\quad x\in \Gamma _{1},\ t>0.
\end{equation}
They proved under the assumption
\begin{equation}
\mu _{2}<\mu _{1}  \label{coeff}
\end{equation}%
 that the solution is exponentially stable. On
the contrary, if (\ref{coeff}) does not hold, they found a sequence of
delays for which the corresponding solution of (\ref{Wave_equation}) will be
unstable. The main approach used in \cite{NP06}, is an observability
inequality obtained with a Carleman estimate. The same results were showed
if both the damping and the delay are acting in the domain. We also recall
the result by Xu, Yung and Li \cite{XYL06}, where the authors proved the
same result as in \cite{NP06} for the one space dimension by adopting the
spectral analysis approach.
We point out that problem (\ref{ondes}) has been already studied by Nicaise and Pignotti in \cite{NPig_2011} for $\mu_1=0$ and a time-varying delay. They find the same condition as the one
used in this paper when $\mu_{1} = 0$ by a different way. However our result 
and our Lyapunov functional  are slightly different here. (See Remark \ref{Remark_Nicaise_work} for more details), and we want to point out that this paper may be viewed
as a continuation of the work  of Nicaise 	and Pignotti  \cite{NPig_2011} in which an additional damping term acts on the boundary and the study of the competition between
these two damping terms is very interesting.

As it has been proved by Datko \cite[Example 3.5]{Dat91}, systems of the
form
\begin{equation}  \label{stron_datko}
\left\{
\begin{array}{ll}
w_{tt}-w_{xx}-aw_{xxt}=0, & x\in(0,1),\,t>0, \vspace{0.3cm} \\
w\left( 0,t\right) =0,\qquad w_{x}\left( 1,t\right) =-kw_{t}\left( 1,t-\tau
\right), & t>0,%
\end{array}%
\right.
\end{equation}
where $a,\,k$ and $\tau$ are positive constants become unstable for an
arbitrarily small values of $\tau$ and any values of $a$ and $k$. In (\ref%
{stron_datko}) and even in the presence of the strong damping $-aw_{xxt}$,
without any other damping, the overall structure can be unstable. This was
one of the main motivations for considering problem (\ref{ondes}). (Of
course the structure of problem (\ref{ondes}) and (\ref{stron_datko}) are
different due to the nature of the boundary conditions in each problem).

Subsequently, Datko et \emph{al} \cite{DLP86} treated the following one
dimensional problem:
\begin{equation}
\left\{
\begin{array}{ll}
u_{tt}(x,t)-u_{xx}(x,t)+2au_{t}(x,t)+a^{2}u(x,t)=0, & 0<x<1,\,\ t>0,\vspace{%
0.3cm} \\
u(0,t)=0, & t>0,\vspace{0.3cm} \\
u_{x}(1,t)=-ku_t(1,t-\tau ), & t>0,%
\end{array}%
\right.  \label{Dakto_system}
\end{equation}%
which models the vibrations of a string clamped at one end and free at the
other end, where $u(x,t)$ is the displacement of the string. Also, the
string is controlled by a boundary control force (with a delay) at the free
end. They showed that, if the positive constants $a$ and $k$ satisfy
\begin{equation*}
k\frac{e^{2a}+1}{e^{2a}-1}<1,
\end{equation*}%
then the delayed feedback system (\ref{Dakto_system}) is stable for all
sufficiently small delays. On the other hand if
\begin{equation*}
k\frac{e^{2a}+1}{e^{2a}-1}>1,
\end{equation*}
then there exists a dense open set $D$ in $(0,\infty)$ such that for each $%
\tau\in D$, system (\ref{Dakto_system}) admits exponentially unstable
solutions.

As a consequence of what we have said before, two main questions naturally
arise here:

\begin{itemize}
\item Is it possible for the damping term $-\Delta u_t$ to stabilize system (%
\ref{ondes}) when the weight of the delay is greater than the weight of the
boundary damping (i.e. when $\mu_2\geq\mu_1$)?

\item Does the particular structure of the problem prevents the instability
result obtained in \cite{Dat91} for problem (\ref{stron_datko})?
\end{itemize}

One of the main purpose of this paper is to give positive answers to the
above two questions. More precisely, we study the asymptotic behavior (as $%
t\rightarrow \infty$) and related decay rates for the corresponding
solutions of system (\ref{ondes}) where the question to be addressed here is
whether the delay term $\mu _{2}u_{t}(x,t-\tau )$ can destroy the stability
of the system, which is exponentially stable in the absence of that delay
\cite{GS082}. As we shall see below, the presence of the strong damping term
$\alpha\Delta u_{t}$ in (\ref{ondes}) plays a decisive role in the stability
of the whole system if (\ref{coeff}) does not hold. Thanks to the energy
method, we built appropriate Lyapunov functionals lead to stability results.

The paper is organized as follows: in the next section, we  prove the
global existence of the solutions by using the Lumer-Phillips' theorem
in the same way as  in \cite{NP06}. In  section 3, we  show that if the weight
of the delay is less than the weight of the damping, then the energy defined by (\ref{energy})
decays exponentially to zero. We also prove that even if
the weight of the delay is greater than the weight of the damping, the
solution still decays to zero exponentially provided that the damping
parameter $\alpha$ satisfies an appropriate condition. Let us mention that
without the damping factor $\alpha$, Nicaise and Pignotti \cite{NP06} proved
the instability of the null stationary solution in the case $\mu_{2} \geq
\mu_{1}$, whereas we will show that if $\mu_{2} \geq \mu_{1}$, by adding a
condition of the form $\alpha > (\mu_{2} - \mu_{1}) B^{2} $ (with $B$ a
constant defined later), we are able to prove the stability of the null
stationary state thanks to a suitable choice of a Lyapunov function.


\section{Well-posedness of Problem (\protect\ref{ondes}).}

In this section we will first transform the delay
boundary conditions by adding a new unknown. Then as in \cite{NP06},
we will use the Lumer-Phillips' theorem to prove the existence and uniqueness
of the solution of  problem (\ref{ondes}).


\subsection{Setup and notations}

 We denote
\begin{equation*}
H_{\Gamma_{0}}^{1}(\Omega) =\left\{u \in H^1(\Omega) /\ u_{\Gamma_{0}} =
0\right\} .
\end{equation*}

We set $\gamma_{1}$ the trace operator from $H_{\Gamma_{0}}^{1}(\Omega) $ on $L^{2}(\Gamma_{1})$ and
$H^{1/2}(\Gamma_{1}) = \gamma_{1}\big(H_{\Gamma_{0}}^{1}(\Omega)\big)$. 
\textcolor{black}{
We denote by $B$ the norm of $\gamma_{1}$ namely:
\begin{equation*}
\forall u \in H_{\Gamma_{0}}^{1}\left( \Omega \right) \,,\, \Vert u\Vert
_{2,\Gamma_{1}}\leq B\Vert \nabla u\Vert_{2} \; .  \label{trace*}
\end{equation*}
}
We recall that $H^{1/2}(\Gamma_{1})$ is
dense in $L^{2}(\Gamma_{1})$ (see  \cite{LM68}).

We denote $E(\Delta,L^{2}(\Omega))= \{u \in H^{1}(\Omega)\, \mbox{such that } \Delta u  \in L^{2}(\Omega)\}$ and recall that for a function
$u \in E(\Delta,L^{2}(\Omega)) \,,\, \displaystyle \frac{\partial u }{\partial \nu} \in  H^{- 1/2}(\Gamma_{1})$ and the next Green's formula is valid (see  \cite{LM68}):
\begin{equation}\label{Green}
 \displaystyle \int_{\Omega} \nabla u(x) \nabla v(x) dx =  \displaystyle \int_{\Omega} - \Delta u(x) v(x) dx + \left\langle\frac{\partial u }{\partial \nu};v \right\rangle_{\Gamma_{1}},\,
 \forall v \in  H_{\Gamma_{0}}^{1}(\Omega),
 \end{equation}
where $\left\langle . ;. \right\rangle_{\Gamma_{1}}$ means the duality pairing between  $H^{- 1/2}(\Gamma_{1})$ and $H^{1/2}(\Gamma_{1})$.

By $( .,.) $ we denote the scalar product in $L^{2}( \Omega)$ i.e. $(u,v)
= \displaystyle \int_{\Omega} u(x) v(x) dx$. Also we mean by $\Vert
.\Vert_{q}$ the $L^{q}(\Omega) $ norm for $1 \leq q \leq \infty$, and by $%
\Vert .\Vert_{q,\Gamma_{1}}$ the $L^{q}(\Gamma_{1}) $ norm.

Throughout the paper, we use the standard notations as in the book \cite{B83} for example.

In order to prove the local existence of the solution of problem (\ref{ondes}%
), we consider the following two cases :

\begin{description}
\item[\textbf{case 1:}] $\mu _{2}<\mu _{1}$. We may define a positive real
number $\xi $ such that:
\begin{equation}
\tau \mu _{2}\leq \xi \leq \tau \left( 2\mu _{1}-\mu _{2}\right) \;.
\label{cond_1}
\end{equation}

\item[\textbf{case 2:}] $\mu_{2} \geq \mu _{1}$. We will suppose that the
damping parameter $\alpha$ verifies:
\begin{equation}  \label{condalpha}
\alpha > (\mu_{2} - \mu_{1}) B^{2} \; .
\end{equation}
\end{description}

In this case, we may define a positive real number $\xi $ satisfying the two
inequalities:
\begin{eqnarray}
\xi &\geq &\tau \mu _{2}\,,  \label{cond_1b} \\
\alpha &>&\left( \frac{\mu _{2}}{2}+\frac{\xi }{2\tau }-\mu _{1}\right)
B^{2}>0\;.  \label{cond_2}
\end{eqnarray}


\subsection{Semigroup formulation of the problem}

In this section, we  prove the global existence and the uniqueness of
the solution of problem (\ref{ondes}).
To overcome the problem of the boundary delay, we introduce, as in \cite{NP06}, the new variable:
\begin{equation}
z\left( x,\rho ,t\right) =u_{t}\left( x,t-\tau \rho \right) ,\ x\in \Gamma
_{1},\ \rho \in \left( 0,1\right) ,\ t>0.  \label{change-variable}
\end{equation}%
Then, we have
\begin{equation}
\tau z_{t}\left( x,\rho ,t\right) +z_{\rho }\left( x,\rho ,t\right) =0,~%
\text{in }\Gamma _{1}\times \left( 0,1\right) \times \left( 0,+\infty
\right) .  \label{equation-z}
\end{equation}%
Therefore, problem (\ref{ondes}) is equivalent to:
\begin{equation}
\left\{
\begin{array}{ll}
u_{tt}-\Delta u-\alpha \Delta u_{t}=0, & x\in \Omega ,\ t>0\,,\\[0.1cm]
\tau z_{t}(x,\rho ,t)+z_{\rho }(x,\rho ,t)=0, & x\in \Gamma _{1},\rho \in(0,1)\,,\,t>0 \,,\\[0.1cm]
u(x,t)=0, & x\in \Gamma _{0},\ t>0 \,,\\[0.1cm]
u_{tt}(x,t)=-\left( \displaystyle\frac{\partial u}{\partial \nu }(x,t)+\alpha \frac{\partial u_{t}}{\partial \nu }(x,t)+
\mu_{1}u_{t}(x,t)+\mu _{2}z(x,1,t)\right), & x\in \Gamma _{1},\ t>0 \,,\\[0.1cm]
z(x,0,t)=u_{t}(x,t), & x\in \Gamma _{1},\ t>0 \,,\\[0.1cm]
u(x,0)=u_{0}(x), & x\in \Omega \,,\\
u_{t}(x,0)=u_{1}(x), & x\in \Omega \,,\\
z(x,\rho ,0)=f_{0}(x,-\tau \rho ), & x\in \Gamma _{1},\,\rho \in (0,1) \ .%
\end{array}%
\right.  \label{wave2}
\end{equation}

The first natural question is the existence of solutions of the problem (\ref{wave2}).
In this section we will give a sufficient condition that guarantees the well-posedness of the problem.

For this purpose, as in \cite{NP06}, we will use a semigroup formulation of the
initial-boundary value problem (\ref{wave2}). 
\textcolor{black}{
If we denote $V:=\left( u,u_{t},\gamma _{1}(u_{t}),z\right)^{T}$, we define the energy space:
\begin{equation*}
\mathscr{H}=H_{\Gamma _{0}}^{1}(\Omega )\times L^{2}\left( \Omega \right) \times L^{2}(\Gamma _{1})\times L^{2}(\Gamma_{1} \times (0,1)).
\end{equation*}%
Clearly, $\mathscr{H}$ is a Hilbert space with respect to the inner product%
\begin{equation*}
\left\langle V_{1},V_{2}\right\rangle _{\mathscr{H}}=\int_{\Omega }\nabla
u_{1}.\nabla u_{2}dx+\int_{\Omega }v_{1}v_{2}dx+\int_{\Gamma _{1}}w_{1}w_{2} d\sigma +\xi \int_{\Gamma _{1}}\int_{0}^{1}z_{1}z_{2}d\rho d\sigma
\end{equation*}%
for $V_{1}=(u_{1},v_{1},w_{1},z_{1})^{T}$, $V_{2}=(u_{2},v_{2},w_{2},z_{2})^{T}$ and $\xi $ is defined by (\ref{cond_1}) or (\ref{cond_1b}).\\
Therefore, if $V_{0} \in \mathscr{H} \mbox{ and } V \in \mathscr{H}$, the problem (\ref{wave2}) is formally 
equivalent to the following abstract evolution equation in the Hilbert space $\mathscr{H}$:
}
\begin{equation}
\left\{
\begin{array}{ll}
V'(t)=  \mathscr{A}V(t) , & t>0, \vspace{0.1cm}\\
V\left( 0\right) =V_{0}, &
\end{array}
\right.  \label{Matrix_problem}
\end{equation}
where $'$ denotes the derivative with respect to time $t$, $V_{0}:=\left( u_{0},u_{1},\gamma_{1}(u_{1}),f_{0}(.,-.\tau)\right)^{T}$
and the operator $\mathscr{A}$ is defined by:
\begin{equation*}
\mathscr{A}\left(
\begin{array}{c}
u\vspace{0.2cm} \\
v\vspace{0.2cm} \\
w\vspace{0.2cm} \\
z
\end{array}
\right) =\left(
\begin{array}{c}
\displaystyle v\vspace{0.2cm} \\
\displaystyle \Delta u+\alpha \Delta v\vspace{0.2cm} \\
\displaystyle -\frac{\partial u}{\partial \nu }-\alpha \frac{\partial v}{\partial \nu }-\mu _{1}v-\mu _{2}z\left( .,1\right) \vspace{0.2cm}\\
\displaystyle -\frac{1}{\tau }z_{\rho }
\end{array}%
\right).
\end{equation*}%

The domain of $\mathscr{A}$ is the set of $V=(u,v,w,z)^{T}$ such that:
\begin{eqnarray}
(u,v,w,z)^{T}\in H_{\Gamma_{0}}^{1}(\Omega)\times H_{\Gamma_{0}}^{1}(\Omega)\times L^{2}(\Gamma_{1})\times L^{2}\left(\Gamma_{1};H^{1}(0,1)\right) ,&& \label{domainA1}\vspace*{0.3cm} \\
\displaystyle u+\alpha v \in E(\Delta,L^{2}(\Omega)) \,,\, \frac{\partial (u+\alpha v)}{\partial \nu } \in L^{2}(\Gamma_{1}), && \label{domainA2} \vspace*{0.3cm}\\
w=\gamma_{1}(v) = \,z(.,0) \text{ on }\Gamma_{1}.&& \label{domainA3}
\end{eqnarray}

The well-posedness of problem (\ref{wave2}) is ensured by:
\begin{theorem}\label{existence_u}
Suppose that $\mu _{2}\geq\mu _{1}$ and $\alpha >(\mu_{2}-\mu _{1})B^{2}$  or $\mu _{2} < \mu _{1}$. Let $V_{0}\in \mathscr{H}$, then there
exists a unique solution $V\in C\left( \mathbb{R}_{+};\mathscr{H}%
\right) $ of problem (\ref{Matrix_problem}). Moreover, if $V_{0}\in %
\mathscr{D}\left( \mathscr{A}\right) $, then
\begin{equation*}
V\in C\left( \mathbb{R}_{+};\mathscr{D}\left( \mathscr{A}\right) \right)
\cap C^{1}\left( \mathbb{R}_{+};\mathscr{H}\right) .
\end{equation*}
\end{theorem}
\begin{proof} To prove Theorem \ref{existence_u}, we use the
Lumer-Phillips' theorem.
For this purpose, we show firstly that the operator $\mathscr{A}$ is
dissipative.
Indeed, let $V=(u,v,w,z)^{T}\in \mathscr{D}\left(\mathscr{A}\right) $. We have
\begin{equation*}
\begin{array}{ll}
\displaystyle \left\langle \mathscr{A}V,V\right\rangle_{\mathscr{H}} =& \displaystyle \int_{\Omega}\nabla u.\nabla v dx + \int_{\Omega }v\left( \Delta u+\alpha \Delta v\right) dx  \\
\displaystyle &+ \displaystyle \int_{\Gamma_{1}}w\left( -\frac{\partial u}{\partial \nu }-\alpha \frac{\partial v}{\partial \nu }-\mu _{1}v-\mu _{2}z\left(\sigma,1\right)
\right) d\sigma -\dfrac{\xi }{\tau }\int_{\Gamma _{1}}\int_{0}^{1}zz_{\rho}d\rho d\sigma . \
\end{array}
\end{equation*}
But since $u+\alpha v \in E(\Delta,L^{2}(\Omega)) \mbox{ and } \displaystyle \frac{\partial (u+\alpha v)}{\partial \nu } \in L^{2}(\Gamma_{1})$,
we may apply Green's formula (\ref{Green}) where the duality pairing $\left\langle . ;. \right\rangle_{\Gamma_{1}}$  is simply the $L^{2}(\Gamma_{1})$ inner product
(because $w = \gamma_{1}(v) \in L^{2}(\Gamma_{1})$) and obtain:
\begin{equation}
\left\langle \mathscr{A}V,V\right\rangle_{\mathscr{H}} =-\mu _{1}\int_{\Gamma _{1}}w^{2}d\sigma -\mu _{2}\int_{\Gamma_{1}}z\left(\sigma,1\right) wd\sigma -\alpha \int_{\Omega }\left\vert \nabla
v\right\vert ^{2}dx-\frac{\xi }{\tau }\int_{\Gamma _{1}}\int_{0}^{1}z_{\rho
}zd\rho dx.  \label{dissipative_A}
\end{equation}
At this point, we have to distinguish the following two cases:\newline
\textbf{Case 1:} We suppose that $\mu _{2}<\mu _{1}$. Let us choose then $%
\xi $ that satisfies inequality (\ref{cond_1}). Using Young's inequality, (%
\ref{dissipative_A}) leads to
\begin{eqnarray*}
&&\left\langle \mathscr{A}V,V\right\rangle _{\mathscr{H}}+\alpha
\int_{\Omega }\left\vert \nabla v\right\vert ^{2}dx+\left( \mu _{1}-\frac{%
\xi }{2\tau }-\frac{\mu _{2}}{2}\right) \int_{\Gamma _{1}}w^{2}d\sigma \\
&&+\left( \frac{\xi }{2\tau }-\frac{\mu _{2}}{2}\right) \int_{\Gamma
_{1}}z^{2}(\sigma ,1)\,d\sigma \leq 0.
\end{eqnarray*}%
Consequently, by using (\ref{cond_1}), we deduce that
\begin{equation}  \label{diss_A}
\left\langle \mathscr{A}V,V\right\rangle _{\mathscr{H}}\leq 0.
\end{equation}

\textbf{Case 2:} We suppose that $\mu _{2}\geq\mu _{1}$ and $\alpha >(\mu
_{2}-\mu _{1})B^{2}$ . Let us choose then $\xi $ that satisfies the two
inequalities (\ref{cond_1b}) and (\ref{cond_2}). Using Young's inequality
and the definition of the constant $B$, we can again prove that the
inequality (\ref{diss_A}) holds. This means that in both cases $\mathscr{A}$
is dissipative.

\vspace*{0.25cm}

Now we  show that  $\lambda I -\mathscr{A}$ is surjective for all $\lambda > 0$.

For $ F=(f_{1},f_{2},f_{3},f_{4})^{T}\in \mathscr{H}$, let $V=(u,v,w,z)^{T}\in \mathscr{D}\left( \mathscr{A}\right) $
solution of
\begin{equation*}
\left( \lambda I-\mathscr{A}\right) V=F,
\end{equation*}%
which is:
\begin{eqnarray}
\lambda u - v &=& f_{1},\vspace{0.2cm}  \label{eqsurj1}\\
\lambda v - \Delta (u + \alpha v ) &=& f_{2}, \label{eqsurj2} \\
\lambda w + \frac{\partial (u+\alpha v)}{\partial \nu } + \mu_{1} v + \mu_{2} z(.,1) &=& f_{3}, \label{eqsurj3} \\
\lambda z + \frac{1}{\tau} z_{\rho}  &=& f_{4}.\label{eqsurj4} \
\end{eqnarray}%

To find $V=(u,v,w,z)^{T}\in \mathscr{D}\left( \mathscr{A}\right) $ solution of the system (\ref{eqsurj1}), (\ref{eqsurj2}), (\ref{eqsurj3}) and (\ref{eqsurj4}), we
proceed as in \cite{NP06}.

Suppose $u$ is determined with the appropriate regularity. Then from  (\ref{eqsurj1}), we get:
\begin{equation}
v=\lambda u-f_{1} \ .  \label{solution_v}
\end{equation}
Therefore, from the compatibility condition on $\Gamma_{1}$, (\ref{domainA3}), we determine $z(.,0)$ by:
\begin{equation}
z(x,0) = v( x) =\lambda u(x) -f_{1}(x) , \ \mbox{\ for } \, x \in \Gamma_{1}.
\label{z_solution}
\end{equation}
Thus, from  (\ref{eqsurj4}), $z$ is the solution of the linear Cauchy problem:
\begin{equation}\label{diffz}
\left\{
\begin{array}{ll}
z_{\rho} = \tau\Big(f_{4}(x) - \lambda z(x,\rho) \Big), & \mbox{ for } x \in \Gamma_{1} \,,\, \rho \in (0,1), \\
z(x,0)  =\lambda u(x) -f_{1}(x).&
\end{array}
\right.
\end{equation}
The solution of the Cauchy problem (\ref{diffz}) is given by:
\begin{equation}\label{z_formula}
z(x,\rho) = \lambda u( x) e^{-\lambda \rho \tau} -f_{1}e^{-\lambda \rho \tau }+\tau e^{-\lambda \rho \tau}\int_{0}^{\rho}f_{4}(x,\sigma) e^{\lambda \sigma \tau }d\sigma
\quad \mbox{ for } x \in \Gamma_{1} \,,\, \rho \in (0,1) .
\end{equation}
So, we have at the point $\rho = 1$,
\begin{equation} \label{z1}
z(x,1) = \lambda u(x)  e^{-\lambda \tau} + z_{1}(x), \quad \mbox{ for } x \in \Gamma_{1}
\end{equation}
with
$$
z_{1}(x) = -f_{1}e^{-\lambda \tau }+\tau e^{-\lambda  \tau}\int_{0}^{1}f_{4}(x,\sigma) e^{\lambda \sigma \tau }d\sigma,
\quad \mbox{ for } x \in \Gamma_{1} .
$$
Since $f_{1} \in H_{\Gamma_{0}}^{1}(\Omega) \mbox{ and } f_{4} \in L^{2}(\Gamma_{1}) \times L^{2}(0,1)$, then $z_{1} \in L^{2}(\Gamma_{1})$.

Consequently, knowing  $u$, we may deduce $v$ by  (\ref{solution_v}), $z$ by (\ref{z_formula}) and using (\ref{z1}), we deduce $w = \gamma_{1}(v)$ by (\ref{eqsurj3}).

In view of equations (\ref{eqsurj2}) and (\ref{eqsurj3}), we set, as in \cite{NPig_2011}, $\overline{u} = u + \alpha v$.  Then, from  (\ref{solution_v}), we have
$$ v = \lambda u - f_{1} = \lambda( \overline{u} - \alpha v) - f_{1} \ .$$
Since $\lambda > 0 \mbox{ and } \alpha >0 \,,\, 1 + \lambda \alpha \neq 0$; thus we have:
\begin{equation}\label{solv}
v = \frac{\lambda}{1 + \lambda \alpha} \overline{u} - \frac{f_{1}}{1 + \lambda \alpha} .
\end{equation}
But since $u = \overline{u} - \alpha v$, we have:
\begin{equation}\label{solu}
u = \frac{1}{1 + \lambda \alpha} \overline{u} + \frac{\alpha}{1 + \lambda \alpha} f_{1} .
\end{equation}
From  equations (\ref{eqsurj2}) and (\ref{eqsurj3}), $\overline{u}$ must satisfy:
\begin{equation} \label{equbar}
\frac{\lambda^{2}}{1 + \lambda \alpha} \overline{u} - \Delta \overline{u} = f_{2} +  \frac{\lambda}{1 + \lambda \alpha} f_{1}, \quad \mbox{in } \Omega
\end{equation}
with the boundary conditions
\begin{eqnarray}
\overline{u} = 0, & \mbox{ on } & \Gamma_{0} \label{ubar0} \\
\frac{\partial \overline{u}}{\partial \nu} = f_{3} - \lambda \gamma_{1}(v) - \mu_{1} \gamma_{1}(v) - \mu_{2} z(.,1), &\mbox{on } & \Gamma_{1} \label{ubar1}
\end{eqnarray}
the last equation at least formally since we don't have yet found the regularity of $\overline{u}$. Replacing $u$ by its expression (\ref{solu}) and inserting it in equation
(\ref{z1}), we get:
$$
z(x,1) = \frac{\lambda}{1 + \lambda \alpha} \overline{u}(x) e^{-\lambda \tau} + \frac{\lambda \alpha }{1 + \lambda \alpha} f_{1}(x) e^{-\lambda \tau} + z_{1}(x), \quad \mbox{ for } x \in \Gamma_{1} .
$$
Using the preceding expression of $z(.,1)$ and the expression of $v$ given by (\ref{solv}), we have:
\begin{equation} \label{dudnu}
\frac{\partial \overline{u}}{\partial \nu} = - \, \frac{\lambda\Big(\mu_{2} e^{-\lambda \tau}
+ (\lambda + \mu_{1}\Big)}{1 + \lambda \alpha} \overline{u} + f(x), \quad \mbox{ for  } x \in \Gamma_{1}
\end{equation}
with
$$
f(x) = f_{3}(x) + \frac{(\lambda + \mu_{1}) - \mu_{2} \lambda \alpha e^{-\lambda \tau}}{1 + \lambda \alpha}  f_{1}(x)  - \mu_{2} z_{1}(x),  \quad \mbox{ for  } x \in \Gamma_{1} \ .
$$
From the regularity of $f_{1} \,,\, f_{2} \,,\, z_{1}$, we get $f \in L^{2}(\Gamma_{1})$.

The variational formulation of problem (\ref{equbar}), (\ref{ubar0}),(\ref{dudnu}) is to find  $\overline{u} \in H_{\Gamma_{0}}^{1}(\Omega)$ such that:
\begin{eqnarray}
\int_{\Omega} \frac{\lambda^{2}}{1 + \lambda \alpha} \overline{u} \omega + \nabla \overline{u} \nabla \omega dx &+ &
\int_{\Gamma_{1}} \frac{\lambda \Big(\mu_{2} e^{-\lambda \tau} + (\lambda + \mu_{1}\Big)}{1 + \lambda \alpha} \overline{u}(\sigma)  \omega(\sigma) d \sigma, \label{varia}\\
&=&
\int_{\Omega}  \left(f_{2} +  \frac{\lambda}{1 + \lambda \alpha} f_{1}\right) \omega  dx + \int_{\Gamma_{1}} f(\sigma) \omega(\sigma) d \sigma,
\notag
\end{eqnarray}
for any $\omega \in H_{\Gamma_{0}}^{1}(\Omega)$. Since $\lambda >0 \,,\, \mu_{1} >0 \,,\, \mu_{2} >0 $, the left hand side of (\ref{varia}) defines a coercive bilinear form on $H_{\Gamma_{0}}^{1}(\Omega)$.
Thus by applying the Lax-Milgram theorem, there exists a unique $\overline{u} \in H_{\Gamma_{0}}^{1}(\Omega)$ solution of (\ref{varia}). Now, choosing
$\omega \in \mathscr{C}_{c}^{\infty}$, $\overline{u}$ is a solution of (\ref{equbar}) in the sense of distribution and therefore $\overline{u} \in E(\Delta,L^{2}(\Omega))$.
Thus using the Green's formula (\ref{Green}) in (\ref{varia}) and exploiting the equation (\ref{equbar}) on $\Omega$, we obtain finally:
$$
\int_{\Gamma_{1}} \frac{\lambda \Big(\mu_{2} e^{-\lambda \tau} + (\lambda + \mu_{1}\Big)}{1 + \lambda \alpha} \overline{u}(\sigma)  \omega(\sigma) d \sigma  +
 \left\langle\frac{\partial \overline{u}}{\partial \nu};\omega \right\rangle_{\Gamma_{1}} = \int_{\Gamma_{1}} f(\sigma) \omega(\sigma) d \sigma \ \forall \omega \in H_{\Gamma_{0}}^{1}(\Omega) \ .
$$
So $\overline{u} \in E(\Delta,L^{2}(\Omega))$ verifies (\ref{dudnu}) and by equation (\ref{solu}) and (\ref{solv}) we recover $u$ and $v$ and thus by (\ref{z_formula}), we obtain $z$ and
finally setting $w = \gamma_{1}(v)$, we have found  $V=(u,v,w,z)^{T}\in \mathscr{D}\left( \mathscr{A}\right) $
solution of $
\left( Id-\mathscr{A}\right) V=F $.

Thus, the proof of Theroem \ref{existence_u}, follows from the Lumer-Phillips' theorem.
\end{proof}


\section{Asymptotic behavior}

\subsection{Exponential stability for $\protect\mu _{2}<\protect\mu _{1}$}

In this subsection, we show that under the assumption $\mu _{2}<\mu
_{1} $, the solution of problem (\ref{wave2}) decays to the null steady
state with an exponential decay rate. For this goal, we  use the
energy method combined with the choice of a suitable Lyapunov functional.

For a positive constant $\xi $ satisfying the strict inequality (\ref{cond_1}%
), (i.e. $<$ instead of $\leq$) we define the functional energy of the
solution of problem (\ref{wave2}) as
\begin{eqnarray}
E(t)=E(t,z,u) &=&\frac{1}{2}\left[ \left\Vert \nabla u(t)\right\Vert
_{2}^{2}+\left\Vert u_{t}(t)\right\Vert _{2}^{2}+\left\Vert
u_{t}(t)\right\Vert _{2,\Gamma _{1}}^{2}\right]  \notag \\
&+&\frac{\xi }{2}\int_{\Gamma _{1}}\int_{0}^{1}z^{2}(\sigma ,\rho ,t)\,d\rho
\,d\sigma \notag\\
&=& \frac{1}{2}E_1(t)+\frac{\xi }{2}\int_{\Gamma _{1}}\int_{0}^{1}z^{2}(\sigma ,\rho ,t)\,d\rho
\,d\sigma, \label{energy}
\end{eqnarray}
where
$$E_{1}(t)= \left\Vert \nabla
u(t)\right\Vert _{2}^{2}+\left\Vert u_{t}(t)\right\Vert _{2}^{2}+\left\Vert
u_{t}(t)\right\Vert _{2,\Gamma _{1}}^{2}. $$
Let us first remark that this energy is greater than the usual one of the
solution of problem (\ref{ondes}), namely $E_1(t)$.\newline
Now, we prove that the above energy $E\left( t\right) $ is a decreasing
function along the trajectories. More precisely,
 we have the following result:

\begin{lemma}
\label{lemme1} \label{dissipativeE1} Assume that $\mu _{1}>\mu _{2}$, then
the energy defined by (\ref{energy}) is a non-increasing positive function
and there exists a positive constant $C$ such that for $(u,z)$ solution of (\ref%
{wave2}), and for any $t \geq 0$, we have:
\begin{equation}
\frac{dE\left( t\right) }{dt}\leq -C\left[\int_{\Gamma_{1}}u_{t}^{2}(%
\sigma,t) \; d\sigma +\int_{\Gamma _{1}}z^{2}(\sigma,1,t) \,d\sigma \right]-
\alpha \int_{\Omega }\left\vert \nabla u_{t}(x,t) \right\vert^{2} \;dx \;.
\label{dissipationE}
\end{equation}
\end{lemma}

\begin{proof}We multiply the first equation in (\ref{wave2}) by $u_{t}$
and perform integration by parts to get:
\begin{equation}  \label{Etmp}
\begin{array}{ll}
\displaystyle \frac{1}{2}\frac{d}{dt}\left[ \left\Vert \nabla u(t)
\right\Vert_{2}^{2}+\left\Vert u_{t}(t) \right\Vert_{2}^{2}+ \left\Vert
u_{t}(t) \right\Vert _{2,\Gamma _{1}}^{2}\right] +\alpha \left\Vert \nabla
u_{t}(t) \right\Vert_{2}^{2} &  \\
\displaystyle +\mu _{1}\left\Vert u_{t}(t) \right\Vert
_{2,\Gamma_{1}}^{2}+\mu_{2}\int_{\Gamma_{1}}
u_{t}(\sigma,t)u_{t}(\sigma,t-\tau) d\sigma =0 \; . &
\end{array}%
\end{equation}
We multiply the third equation in (\ref{wave2}) by $\xi z$, integrate the
result over $\Gamma _{1}\times (0,1)$, we obtain:
\begin{eqnarray}
\frac{\xi }{\tau }\int_{\Gamma _{1}}\int_{0}^{1}z_{\rho }z(\sigma,\rho,t) \,
d\rho \, d\sigma &=&\frac{\xi }{2\tau}\int_{\Gamma _{1}} \int_{0}^{1}\frac{%
\partial }{\partial \rho }z^{2}(\sigma,\rho ,t) \, d\rho \, d\sigma  \notag
\\
&=&\frac{\xi }{2\tau}\int_{\Gamma_{1}}\left(
z^{2}(\sigma,1,t)-z^{2}(\sigma,0,t) \right) d\sigma \;.
\label{Energy_second_equation}
\end{eqnarray}
Using the definition (\ref{change-variable}) of $z$ in the equality (\ref{Etmp}) and using the same technique as in the  first step  of the proof of
Theorem \ref{existence_u}, where we proved that $\mathscr{A}$ is dissipative, inequality (\ref{dissipationE}) holds.
\end{proof}
The asymptotic stability result reads as follows:

\begin{theorem}
\label{exponential1} Assume that $\mu _{2}<\mu _{1}$. Then there exist two
positive constants $C$ and $\gamma $ independent of $t$ such that for $(u,z)$
solution of problem (\ref{wave2}), we have:
\begin{equation}
E(t)\leq Ce^{-\gamma t},\quad \forall \,t\geq 0\;.  \label{decay1}
\end{equation}
\end{theorem}
\begin{proof} The proof of Theorem \ref{exponential1} relies on the
construction of a Lyapunov functional.

For a small positive constant $\varepsilon $ to be chosen later, we define:
\begin{eqnarray}
L(t) =E(t)&+&\varepsilon \int_{\Omega }u(x,t) u_{t}(x,t) \; dx + \varepsilon
\int_{\Gamma_{1}}u(\sigma,t)u_{t}(\sigma,t) \;d\sigma  \notag \\
&+&\frac{\varepsilon \alpha }{2}\int_{\Omega }\left\vert\nabla
u(x,t)\right\vert^{2} \;dx  \label{Lt} \\
&+& \textcolor{black}{\varepsilon \xi \int_{\Gamma_{1}}\int_{0}^{1}e^{-2\tau \rho}z^{2}(\sigma,\rho ,t)\;d\rho \;d\sigma} .  \notag
\end{eqnarray}%
Let us say that the introduction of the last term in the Lyapunov functional $L$ is inspired by the work of Nicaise and Pignotti
\cite{NPig_2011}.

It is straightforward to see that for $\varepsilon > 0$, $L(t)$ and $E(t)$
are equivalent in the sense that there exist two positive constants $\beta
_{1}$ and $\beta _{2}>0$ depending on $\varepsilon $ such that for all $%
t\geq 0$%
\begin{equation}
\beta_{1}E(t)\leq L(t)\leq \beta_{2}E(t) \; .  \label{equivalEL1}
\end{equation}%
By taking the time derivative of the function $L$ defined by (\ref{Lt}),
using the equations in problem (\ref{wave2}), several integration by parts, and
exploiting  (\ref{dissipationE}), we get:
\begin{eqnarray}
\frac{dL(t)}{dt} \leq&-&C\left[ \int_{\Gamma _{1}}u_{t}^{2}(\sigma,t)
\;d\sigma +\int_{\Gamma _{1}}z^{2}(\sigma,1,t) \;d\sigma\right]  \notag \\
&-&\alpha \Vert \nabla u_{t}\Vert_{2}^{2} - \varepsilon \Vert \nabla
u\Vert_{2}^{2} + \varepsilon \Vert u_{t}\Vert_{2}^{2} +\varepsilon \Vert
u_{t}\Vert_{2,\Gamma_{1}}^{2}  \notag \\
&-&\varepsilon \mu_{1}\int_{\Gamma_{1}}u_{t}(\sigma,t) u(\sigma,t) \;d\sigma
- \varepsilon \mu _{2}\int_{\Gamma_{1}}z(\sigma,1,t) u( \sigma,t) d\sigma
\label{dLdt1} \\
&+&\textcolor{black}{\varepsilon \xi\, \frac{d}{dt}\left(
\int_{\Gamma_{1}}\int_{0}^{1}e^{-2 \tau \rho}z^{2}(\sigma,\rho,t) \; d\rho \;
d\sigma \right)} .  \notag
\end{eqnarray}

\textcolor{black}{
By using the second equation in (\ref{wave2}), the last term in (\ref{dLdt1}%
) can be treated as follows:
\begin{eqnarray*}
\varepsilon \xi\,\frac{d}{dt}\left(
\int_{\Gamma_{1}}\int_{0}^{1}e^{-2\tau\rho }z^{2}(\sigma,\rho ,t) \;d\rho
\;d\sigma \right) 
&=&-\frac{ 2 \varepsilon \xi }{\tau}\int_{\Gamma_{1}}\int_{0}^{1}e^{-2 \tau\rho }z(\sigma,\rho ,t)z_{\rho}(\sigma,\rho ,t)\;d\rho \;d\sigma \\
&=&-\frac{\varepsilon \xi }{\tau}\int_{\Gamma_{1}}\int_{0}^{1}e^{-2 \tau\rho }\frac{\partial }{\partial
\rho }z^{2}(\sigma,\rho,t)\;d\rho \;d\sigma .
\end{eqnarray*}%
Then, by using an integration by parts and the definition of $z$, the above formula can be rewritten
as:
\begin{eqnarray}
\varepsilon \xi \frac{d}{dt}\left( \int_{\Gamma_{1}}\int_{0}^{1}e^{-2\tau\rho}z^{2}(\sigma,\rho,t)\;d\rho \;d\sigma \right)  
&=&-\frac{\varepsilon \xi }{\tau}e^{-2\tau}\int_{\Gamma_{1}}z^{2}(1,\rho,t)\;d\rho \;d\sigma + 
\frac{\varepsilon \xi }{\tau }\int_{\Gamma_{1}}u_{t}^{2}(\sigma,t)\;d\sigma  \label{dLdt_term2} \\
&&-2 \varepsilon \xi \int_{\Gamma_{1}}\int_{0}^{1}e^{-2\tau \rho }z^{2}(\sigma,\rho ,t)\;d\rho
\;d\sigma \; \notag.  
\end{eqnarray}
}
Applying Young's inequality, and the trace inequality, we obtain, for any $%
\delta >0$:
\begin{eqnarray}
\left\vert \int_{\Gamma_{1}}u_{t}ud\sigma \right\vert &\leq &\delta \Vert
u\Vert_{2,\Gamma _{1}}^{2}+\frac{1}{4\delta }\Vert u_{t}\Vert
_{2,\Gamma_{1}}^{2}  \notag \\
&\leq &\delta B^{2}\Vert \nabla u\Vert _{2}^{2}+\frac{1}{4\delta }\Vert
u_{t}\Vert_{2,\Gamma _{1}}^{2} .  \label{Young1}
\end{eqnarray}%
Similarly, we have
\begin{equation}
\left\vert \int_{\Gamma _{1}}z(\sigma,1,t) u(\sigma,t) \;d\sigma \right\vert
\leq \frac{1}{4\delta}\int_{\Gamma _{1}}z^{2}(\sigma,1,t) \;d\sigma +\delta
B^{2}\Vert \nabla u\Vert _{2}^{2}  .  \label{Young2}
\end{equation}%
Inserting (\ref{dLdt_term2}),~(\ref{Young1}) and (\ref{Young2}) into (\ref%
{dLdt1}) and using 
\textcolor{black}{
Poincar\'e's inequality for $u_{t}$, in which we denote $C(\Omega)$ the Poincar\'e's constant, namely :
\begin{equation*}
\forall w \in H_{\Gamma_{0}}^{1}\left( \Omega \right) \,,\, \Vert w \Vert_{2} \leq C(\Omega) \Vert \nabla w \Vert_{2} \;
\end{equation*}
}
we have:
\begin{eqnarray}
\frac{dL(t)}{dt} &\leq &-\left[ C-\varepsilon \left(1+\frac{\xi }{\tau }+\frac{\mu _{1}}{4\delta }\right) \right] \Vert u_{t}\Vert _{2,\Gamma_{1}}^{2}
\notag \\
&&-\left[ C-\varepsilon \left( \frac{\xi }{\tau }e^{-2\tau}+\frac{\mu _{2}}{4\delta }\right) \right] \int_{\Gamma_{1}}z^{2}(\sigma,1,t) \;d\sigma  \notag
\\
&&-\left( \alpha -\varepsilon 
\textcolor{black}{
C(\Omega)^{2}
}\right) \Vert \nabla u_{t}\Vert_{2}^{2}-
\varepsilon \Bigl( 1-B^{2}\delta \left( \mu _{1}+\mu_{2}\right)\Bigl) \Vert \nabla u\Vert_{2}^{2}  \notag \\
&&\textcolor{black}{-2 \varepsilon \xi \int_{\Gamma_{1}}\int_{0}^{1}e^{-2\tau \rho }z^{2}(\sigma,\rho ,t)\;d\rho \;d\sigma} .  \label{dLdt_2}
\end{eqnarray}
%
We choose now $\delta $ small enough in (\ref{dLdt_2}) such that
\begin{equation*}
\delta <\frac{1}{B^{2}\left( \mu _{1}+\mu _{2}\right) } \;.
\end{equation*}
Once $\delta $ is fixed, using once again Poincar\'{e}'s inequality in (\ref%
{dLdt_2}), we may pick $\varepsilon $ small enough to obtain the existence
of $\eta >0$, such that:
\begin{equation}
\frac{dL(t)}{dt}\leq -\eta \varepsilon E(t) ,\quad \forall t\geq 0\;.
\label{dLdt_3}
\end{equation}%
On the other hand, by virtue of (\ref{equivalEL1}), setting $\gamma =-\eta
\varepsilon /\beta _{2}$, the last inequality becomes:
\begin{equation}
\frac{dL(t)}{dt}\leq -\gamma L(t)\;,\quad \forall t\geq 0\; .  \label{diffineq}
\end{equation}%
Hence, integrating the previous differential inequality (\ref{diffineq}) between
$0$ and $t$, we get
\begin{equation*}
L(t)\leq C_{\ast}e^{-\gamma t}\;,\quad \forall t\geq 0,
\end{equation*}
for some positive constant $C_{\ast}$.

Consequently, by using (\ref{equivalEL1}) once again, we conclude that it
exists $C > 0$ such that:
\begin{equation*}
E(t)\leq Ce^{-\gamma t}\;,\quad \forall t\geq 0\;.
\end{equation*}%
This completes the proof of Theorem \ref{exponential1} .
\end{proof}


\subsection{Exponential stability for $\protect\mu _{2}>\protect\mu _{1}$
and $\protect\alpha>(\protect\mu _{2} - \protect\mu _{1})B^{2}$}

As, we have said in the Introduction, and it is
clearly observed in Theorem \ref{exponential1}, that the strong internal
damping compensates the destabilizing effect of the delay in the boundary
condition.

In this section, we assume that $\mu_{2} > \mu _{1}$ and $\alpha > (\mu _{2}
- \mu _{1})B^{2}$. As we will see, we cannot directly perform the same proof
as for the case where $\mu_{2} \leq \mu _{1}$, since the boundary delay term
$\mu_{2} u_{t}(x,t-\tau)$ is greater than the normal one $\mu_{1} u_{t}(x,t)$,
i.e. ($\mu_2\geq \mu_1$). So we have to control this term by the damping term $\alpha \Delta u_{t}$
in the equation.
\begin{remark}\label{Remark_mu_1_mu_2}
In the case 2, namely  $\mu_{2} > \mu _{1}$ the condition  $\alpha > (\mu_{2} - \mu _{1})B^{2}$ permits us to find $\xi$ satisfying (\ref{cond_1b})-(\ref{cond_2}).
This choice of  $\xi$ is essential in the proofs of  Lemma \ref{lemme2} and  Theorem \ref{exponential2} below.
\end{remark}
\begin{lemma}
\label{lemme2} \label{dissipativeE2} Assume that $\mu_{2} > \mu _{1}$ and $%
\alpha > (\mu _{2} - \mu _{1})B^{2}$. For any $\xi$ satisfying (\ref{cond_1b}%
)-(\ref{cond_2}), the energy defined by (\ref{energy}) is a non-increasing
positive function and there exists a positive constant $\kappa$ such that
for $(u,z)$ solution of (\ref{wave2}), and for any $t \geq 0$, we have:
\begin{equation}
\frac{dE(t)}{dt}\leq -\kappa \left[ \int_{\Gamma_{1}}z^{2}(\sigma,1,t)
\,d\sigma +\int_{\Omega }\left\vert \nabla u_{t}(x,t) \right\vert ^{2}\, dx%
\right] \;.  \label{dEdt_2}
\end{equation}
\end{lemma}

\begin{proof} Let us first recall from
(\ref{Etmp}) and (\ref{Energy_second_equation}) the following identity
\begin{eqnarray}
\frac{dE(t)}{dt} &=&-\alpha \int_{\Omega }\left\vert \nabla
u_{t}\right\vert^{2}dx-\left(\mu_{1}-\frac{\xi }{2\tau}\right)
\int_{\Gamma_{1}}u_{t}^{2}(\sigma,t) \;d\sigma  \notag \\
&&-\frac{\xi}{2\tau}\int_{\Gamma_{1}}z^{2}(\sigma,1,t) \;d\sigma -
\mu_{2}\int_{\Gamma_{1}}u_{t}(\sigma,t) z(\sigma,1,t) \;d\sigma \;.
\label{dEdt_3}
\end{eqnarray}%
Now, using Young's inequality, then (\ref{dEdt_3}) takes the form:
\begin{eqnarray}
\frac{dE(t)}{dt} &\leq&-\alpha \int_{\Omega}\left\vert \nabla
u_{t}\right\vert^{2} \; dx- \left(\mu_{1}-\frac{\xi}{2\tau}-\frac{\mu_{2}}{2}%
\right) \int_{\Gamma_{1}}u_{t}^{2}(\sigma,t) \;d\sigma  \notag \\
&&-\left( \frac{\xi}{2\tau}-\frac{\mu_{2}}{2}\right)
\int_{\Gamma_{1}}z^{2}(\sigma,1,t) \; d\sigma .  \label{dEdt_4}
\end{eqnarray}%
Since,
\begin{equation*}
\mu_{1}-\frac{\xi}{2\tau}-\frac{\mu _{2}}{2} < 0 ,
\end{equation*}
then, using the trace inequality, we obtain:
\begin{eqnarray*}
\frac{dE(t)}{dt} &\leq&\left(-\alpha - B^{2} \left(\mu_{1}-\frac{\xi}{2\tau}-%
\frac{\mu_{2}}{2}\right)\right) \int_{\Omega}\left\vert \nabla
u_{t}\right\vert^{2} \; dx \\
&&-\left( \frac{\xi}{2\tau}-\frac{\mu_{2}}{2}\right)
\int_{\Gamma_{1}}z^{2}(\sigma,1,t) \; d\sigma .
\end{eqnarray*}
Using the two inequalities (\ref{cond_1b})-(\ref{cond_2}) that $\xi$ satisfy
, we may find $\kappa > 0$ such that the inequality (\ref{dEdt_2}) holds.
\end{proof}
We can now state that under the same assumption as in Lemma \ref{lemme2},
the system (\ref{wave2}) is also exponentially stable. The second stability
result reads as follows:

\begin{theorem}
\label{exponential2} Assume that $\mu_{2} > \mu _{1}$ and $\alpha > (\mu
_{2} - \mu _{1})B^{2}$. For any $\xi$ satisfying (\ref{cond_1b})-(\ref%
{cond_2}), there exist two positive constants $\overline{C}$ and $\overline{%
\gamma}$ independent of $t$ such that for $(u,z)$ solution of problem (\ref%
{wave2}), we have:
\begin{equation}
E(t)\leq \overline{C}e^{-\overline{\gamma}t},\ \forall \,t\geq 0\;.
\label{decay_2}
\end{equation}
\end{theorem}
{\it Proof  of Theorem \ref{exponential2}}. We use the same Lyapunov function as in the  previous section, namely, for a small positive constant 
$\varepsilon $ to be chosen later, we define:
\begin{eqnarray}
L(t) =E(t)&+&\varepsilon \int_{\Omega }u(x,t) u_{t}(x,t) \; dx + \varepsilon
\int_{\Gamma_{1}}u(\sigma,t)u_{t}(\sigma,t) \;d\sigma  \notag \\
&+&\frac{\varepsilon \alpha }{2}\int_{\Omega }\left\vert\nabla
u(x,t)\right\vert^{2} \;dx \notag \\
&+&\textcolor{black}{\varepsilon \xi \int_{\Gamma
_{1}}\int_{0}^{1}e^{-2\tau \rho}z^{2}(\sigma,\rho ,t)\;d\rho \;d\sigma .}  \notag
\end{eqnarray}%

By taking the time derivative of the function $L$,
using the equations in problem (\ref{wave2}), several integration by parts, and
exploiting (\ref{dEdt_2}), we get:
\begin{eqnarray}
\frac{dL(t)}{dt} \leq&-&\kappa \left[ \int_{\Gamma_{1}}z^{2}(\sigma,1,t)
\,d\sigma +\int_{\Omega }\left\vert \nabla u_{t}(x,t) \right\vert ^{2}\, dx%
\right]  \notag \\
&-&\alpha \Vert \nabla u_{t}\Vert_{2}^{2} - \varepsilon \Vert \nabla
u\Vert_{2}^{2} + \varepsilon \Vert u_{t}\Vert_{2}^{2} +\varepsilon \Vert
u_{t}\Vert_{2,\Gamma_{1}}^{2}  \notag \\
&-&\varepsilon \mu_{1}\int_{\Gamma_{1}}u_{t}(\sigma,t) u(\sigma,t) \;d\sigma
- \varepsilon \mu _{2}\int_{\Gamma_{1}}z(\sigma,1,t) u( \sigma,t) d\sigma
\label{dLdt2} \\
&+&\textcolor{black}{\varepsilon \xi\, \frac{d}{dt}\left(
\int_{\Gamma_{1}}\int_{0}^{1}e^{-2\tau \rho }z^{2}(\sigma,\rho,t) \; d\rho \;
d\sigma \right) . } \notag
\end{eqnarray}
Let us remark  that equality (\ref{dLdt_term2}), and inequalities (\ref{Young1}) and (\ref{Young2}) used in the the proof of Theorem \ref{exponential1} are still valid.
Inserting  (\ref{dLdt_term2}),~(\ref{Young1}) and (\ref{Young2}) into (\ref{dLdt2}) and using Poincar\'e's inequality, we have:
\begin{eqnarray}
\frac{dL(t)}{dt} &\leq &-\left[ \kappa -\varepsilon \left( \frac{\xi }{\tau }%
e^{-2\tau}+\frac{\mu _{2}}{4\delta }\right) \right] \int_{\Gamma_{1}}z^{2}(%
\sigma,1,t)\; d\sigma  \notag \\
&&-\left( \kappa -\varepsilon \left( B^{2}+1+\frac{\textcolor{black}{B^{2}}\xi }{\tau }+\frac{\mu_{1}%
}{4\delta }\right) \right) \Vert \nabla u_{t}\Vert_{2}^{2}  \notag \\
&&-\varepsilon \left( 1-\textcolor{black}{C(\Omega)}^{2}\delta \left( \mu_{1}+\mu _{2}\right)
\right)\Vert \nabla u\Vert_{2}^{2}  \label{dL2dt} \\
&&\textcolor{black}{-2 \varepsilon \xi 
\int_{\Gamma_{1}}\int_{0}^{1}e^{-2\tau\rho }z^{2}(\sigma,\rho ,t)\;d\rho
\;d\sigma }\;.  \notag
\end{eqnarray}
The remaining part of the proof is similar to the one of the proof of
Theorem \ref{exponential1}: by choosing firstly $\delta$ and then $\varepsilon$, we may find $\overline{\gamma} >0$ independent of $t$ such
that:
\begin{equation*}
\frac{dL(t)}{dt}\leq -\overline{\gamma} L(t)\;,\; \forall t\geq 0\; .
\end{equation*}
This inequality permits us to conclude the proof of Theorem \ref%
{exponential2}.
\findem

\begin{remark}\label{Remark_Nicaise_work}
 After our work had been submitted, we noticed that a similar problem has been already studied by
  Nicaise and Pignotti \cite{NPig_2011}. However our problem is slightly different from the one considered in \cite{NPig_2011}:
  \begin{itemize}
    \item First, we have the extra term $\mu_1 u_t$ on the boundary conditions, which makes the stability analysis independent of the strong damping $\alpha\Delta u_t,$
    for $\mu_1>\mu_2$. This is not the case in \cite{NPig_2011}. See the assumption (2.56) in \cite{NPig_2011}. Moreover, the paper \cite{NPig_2011} does not cover the case $\alpha=0$ at all.
     \item Secondly, our Lyapunov functional is different from the one used in \cite{NPig_2011}. We choose to use this one to 
    have a strong control of the boundary delay term.
  \end{itemize}
Finally, let us remark that   our assumption $(\mu_2-\mu_1)B^2<\alpha$ is exactly the same  than the one obtained in \cite[Condition (2.56)]{NPig_2011}, when $\mu_{1} = 0$.
So this work can be viewed as a continuation of the works of Nicaise and Pignotti  \cite{NPig_2011}
 
\end{remark}

\begin{ack}
The second author was supported by MIRA 2007 project of the
R\'egion Rh\^one-Alpes. This author wishes to thank Univ. de Savoie of
Chamb\'ery for its kind hospitality. Moreover, the two authors wish to thank the
referees  for their  remarks and the careful reading of the proofs presented in this paper.
\end{ack}

\bibliographystyle{siam}

\end{document}